\documentclass[11pt]{article}
    \usepackage{url}
    \usepackage{verbatim}
    \usepackage[titletoc]{appendix}
    \usepackage{graphicx}
	\usepackage{amsmath}
	\usepackage{amsthm}
	\usepackage{amsfonts}
	\usepackage{graphicx}
    \usepackage{nicefrac}
    \usepackage{longtable}
    \usepackage{color}
    \usepackage{graphicx, amssymb, graphics, bm}

    \newcommand{\be}{\begin{equation}}
    \newcommand{\ee}{\end{equation}}

    \newcommand{\nrm}[1]{\left\| #1 \right\|}

    \def\g{\mbox{\boldmath $g$}}

    \newcommand\dt {{\Delta t}}
     \def\0{\mbox{\boldmath $0$}}

\renewcommand{\d}{\mathrm{d}}

\renewcommand{\i}{\mathrm{i}}
\newcommand{\R}{\mathbb{R}}

\newcommand{\Z}{\mathbb{Z}}

\newcommand{\bx}{{\bf x}}
\newcommand{\by}{{\bf y}}

\newcommand{\hL}{\mathcal{L}}

\newcommand{\Mh}{\mathcal{M}_h}

\newcommand{\ds}{\displaystyle}
\newcommand{\eps}{\varepsilon}
\newcommand{\os}{\,\textcircled{$*$}\,}
\newcommand{\<}{\langle}
\renewcommand{\>}{\rangle}
\newcommand{\Hp}{H_\text{\rm per}}

\newcommand{\Cp}{C_\text{\rm per}}

    \def\g{\mbox{\boldmath $g$}}

\newtheorem{theorem}{Theorem}[section]
\newtheorem{corollary}[theorem]{Corollary}
\newtheorem{lemma}[theorem]{Lemma}

\newtheorem{remark}{Remark}

\title{Convergence analysis for a stabilized linear semi-implicit numerical scheme for the nonlocal Cahn--Hilliard equation}

	\author{Xiao Li\thanks{Department of Applied Mathematics, The Hong Kong
Polytechnic University, Hung Hom, Kowloon, Hong Kong ({\tt  xiao1li@polyu.edu.hk})}
        \and
Zhonghua Qiao\thanks{Corresponding author. Department of Applied Mathematics, The Hong Kong
Polytechnic University, Hung Hom, Kowloon, Hong Kong ({\tt  zqiao@polyu.edu.hk})}
        \and
Cheng Wang\thanks{Mathematics Department, The University of Massachusetts, North Dartmouth, MA, USA
({\tt cwang1@umassd.edu})}}

\date{}

\begin{document}

\maketitle

\begin{abstract}
In this paper, we provide a detailed convergence analysis for a first order stabilized linear semi-implicit numerical scheme for the nonlocal Cahn--Hilliard equation, which follows from consistency and stability estimates for the numerical error function. Due to the complicated form of the nonlinear term, we adopt the discrete $H^{-1}$ norm for the error function to establish the convergence result. In addition, the energy stability obtained in [Du et al., \emph{J. Comput. Phys.}, 363:39--54, 2018] requires an assumption on the uniform $\ell^\infty$ bound of the numerical solution
and such a bound is figured out in this paper by conducting the higher order consistency analysis.
Taking the view that the numerical solution is indeed the exact solution with a perturbation,
the error function is $\ell^\infty$ bounded uniformly under a loose constraint of the time step size,
which then leads to the uniform maximum-norm bound of the numerical solution.

\end{abstract}

	

\section{Introduction}

In this paper, our primary purpose is to develop a detailed convergence analysis
of a stabilized linear semi-implicit numerical scheme for the nonlocal Cahn--Hilliard (NCH) equation taking the form
\cite{bates06a,du18,guan14a}
\begin{equation}
\label{nonlocalCH}
u_t=\Delta(u^3-u+\eps^2\hL u),\quad(\bx,t)\in\Omega\times(0,T],
\end{equation}
where $u=u(\bx,t)$ is the unknown function
subject to 
the periodic boundary condition.
Here, $\Omega=\prod\limits_{i=1}^d(-X_i,X_i)$ is a rectangular domain in $\R^d$,
$T>0$ is the terminal time, $\eps>0$ is an interfacial parameter,
and $\hL$ is a nonlocal linear operator defined by
\begin{equation}
\hL:v(\bx)\mapsto\int_\Omega J(\bx-\by)(v(\bx)-v(\by))\,\d\by,
\end{equation}
where $J$ is a kernel function satisfying following conditions \cite{du18}:
\begin{flushleft}
(a) $J(\bx)\ge0$ for any $\bx\in\Omega$;\\
(b) $J$ is even, i.e., $J (\bx) = J (-\bx)$ for any $\bx \in \R^d$; \\
(c) $J$ is $\Omega$-periodic;\\
(d) $\ds\frac{1}{2}\int_\Omega J(\bx)|\bx|^2\,\d \bx=1$,
\end{flushleft}
where the condition (d) means that the kernel has a finite second order moment in $\Omega$. In fact, $J$ could be taken as a radial function over the domain $\Omega=\prod\limits_{i=1}^d(-X_i,X_i)$, with exponential decay at the boundary, such as $J (\bx) = \alpha \exp \left(- \frac{|\bx|^2}{\sigma^2} \right)$ (for a small $\sigma$), and a periodic extension is made to $\R^d$, so that both (b) and (c) are satisfied. The NCH equation \eqref{nonlocalCH} can be viewed as the $H^{-1}$ gradient flow
with respect to the free energy functional
\begin{equation}
\label{nonlocalenergy2}
E(u)=\int_\Omega F(u(\bx))\,\d \bx+\frac{\eps^2}{2}(\hL u,u)_{L^2},
\end{equation}
with $F(u)=\frac{1}{4}(u^2-1)^2$, or equivalently, by using the condition (b),
\begin{equation}
\label{nonlocalenergy1}
E(u)=\int_\Omega\Big(F(u(\bx))+\frac{\eps^2}{4}\int_\Omega J(\bx-\by)(u(\bx)-u(\by))^2\,\d\by\Big)\,\d\bx.
\end{equation}
The second term in \eqref{nonlocalenergy1} usually represents the interaction energy,
describing the long-range interactions between atoms at different sites,
and the kernel $J$ measures the strength of interactions.
Using the Taylor formula, the periodicity of $u$, and the conditions (b)-(d) of $J$, one can show that
$$\frac{\eps^2}{4}\int_\Omega J(\bx-\by)(u(\bx)-u(\by))^2\,\d\by\approx\frac{\eps^2}{2}|\nabla u(\bx)|^2,$$
which suggests that the classic Cahn--Hilliard equation \cite{cahn58}
\begin{equation}
\label{localCH}
u_t=\Delta(u^3-u-\eps^2\Delta u),
\end{equation}
corresponding to the local energy functional
\begin{equation}
E_{\text{local}}(u)=\int_\Omega\Big(F(u(\bx))+\frac{\eps^2}{2}|\nabla u(\bx)|^2\Big)\,\d \bx,
\end{equation}
is an approximation of the NCH equation \eqref{nonlocalCH}
under the assumption that the interaction exists only in a very short range.

If $J$ is further integrable, then $J*1=\ds\int_\Omega J(\bx)\,\d \bx>0$ is a positive constant and
\begin{equation}
\hL v=(J*1)v-J*v,
\end{equation}
where
$$(J*v)(\bx)=\int_\Omega J(\bx-\by)v(\by)\,\d \by=\int_\Omega J(\by)v(\bx-\by)\,\d \by$$
is exactly the periodic convolution \cite{guan14a}.
In this case, the NCH equation \eqref{nonlocalCH} can be written as
$$u_t=\nabla\cdot(a(u)\nabla u)-\eps^2\Delta J*u,$$
where $a(u)=3u^2-1+\eps^2J*1$ is referred as the diffusive mobility.
If
\begin{equation}
\label{assump_gamma}
\gamma_0:=\eps^2J*1-1>0,
\end{equation}
which gives $a(u)>0$, then the equation \eqref{nonlocalCH} becomes diffusive and the solution becomes regular in time;
otherwise, the solution may exhibit some singular behaviors.
Throughout this paper, we always assume that the kernel $J$ is integrable with the condition \eqref{assump_gamma} held.

As one of typical systems of the phase field models,
the classic Cahn--Hilliard equation \eqref{localCH} has been successfully used to model
phase transitions occurring in mixtures of small molecules
and some other interface problems involving mass-conserved order parameters.
Recently, the NCH equation \eqref{nonlocalCH} has attracted increasing attentions
and been used in various fields ranging from materials science to finance and image processing.
For instance, in materials science, the NCH equation and other related equations
arise as mesoscopic models of interacting particle systems \cite{archer04b,hornthrop01}
and are taken to model phase transitions \cite{fife03};
in the dynamic density functional theory \cite{archer04a,archer04b},
the interaction kernel is the two-particle direct correlation function
and the solution represents the mesoscopic particle density.
In the theoretical level,
the well-posedness of the NCH equations equipped with Neumann or Dirichlet boundary condition
were investigated by Bates and Han \cite{bates05b,bates05a} by assuming the integrability of the kernel.
Du et al. \cite{du12a} developed a general framework of nonlocal diffusion problems and
a number of examples ranging from continuum mechanics to graph theory were showed to be special cases of the proposed framework.
For more details on theoretical investigations, see \cite{bates06a,bates06b,fife03,gajewski03,GAL20175253,giacomin98}.

There have been several works on numerical analysis for nonlocal models.
For a class of nonlocal diffusion models with variable boundary conditions,
finite difference and finite element approximations were addressed in \cite{tian13,tian14,zhou10}.
For the nonlocal Allen--Cahn equation, the $L^2$ gradient flow with respect to \eqref{nonlocalenergy1},
Bates et al. \cite{bates09} developed an $L^\infty$ stable and convergent finite difference scheme by treating the nonlinear and nonlocal terms explicitly
and Du et al. \cite{du16} analyzed the spectral-Galerkin approximations.
In addition, the maximum principle preserving property has been established for the exponetial time differencing (ETD) schemes in a more recent work~\cite{du19}.
For the NCH equation,
an important fact is that the exact solution decreases the energy in time
due to the energetic variational structure of the underlying model,
so it is highly desirable to develop numerical algorithms inheriting this property of energy stability at the discrete level.
Energy stability has been widely investigated for numerical schemes of a family of classic PDE-based phase field models,
such as convex splitting schemes \cite{eyre98,qiao14,wise09}, stabilized schemes \cite{shen10,xu06}, and so on.
The application of similar analysis for nonlocal phase field models are still full of challenges
due to the lack of the higher order diffusion term.
Guan et al. \cite{guan17a,guan14b,guan14a} constructed
convex splitting schemes for the NCH equation and nonlocal Allen--Cahn equation
by treating the nonlinear term implicitly and putting the nonlocal term into the explicit part.
The proposed scheme allows one to evaluate the nonlocal term explicitly only once at each time step,
but iterations are inevitable due to the nonlinearity of the scheme.

In order to avoid the nonlinear iteration, in a recent work \cite{du18},
a linear semi-implicit scheme has been developed by using the stabilizing approach.
The linear nonlocal term is set in the implicit level
and solved efficiently in the frequency space by using the fast Fourier transform (FFT) technique
due to the linearity of the resulted fully discrete system.
The first order stabilized linear semi-implicit (SSI1) scheme given in \cite{du18} reads
\begin{equation}
\label{scheme-1st-1}
\frac{u^{n+1} - u^n}{\dt} = \Delta_N \big[ (u^n)^3 - u^n + A (u^{n+1} - u^n) + \eps^2 \hL_N u^{n+1} \big],
\end{equation}
and the energy stability has been proved, that is,
$E_N(u^{n+1})\le E_N(u^n)$ if the stabilizing constant $A$ satisfies
\begin{equation}
\label{condition-A-0}
A \ge \frac{1}{2} \| u^{n+1} \|_\infty^2 + \| u^n \|_\infty^2 - \frac{1}{2}.
\end{equation}
Here, $E_N$, $\Delta_N$ and $\hL_N$ are the spatially discretized forms of the operators $E$, $\Delta$ and $\hL$, respectively,
and their precise definitions will be given in the next section.
Notice that the infinity-norms of the numerical solutions at time steps $t_n$ and $t_{n+1}$
have been involved on the right hand side of \eqref{condition-A-0}.
However, such a lower bound for constant $A$ has not been justified.

We aim to justify the lower bound of $A$ in this paper.
A direct analysis provided in \cite{LiD2017, LiD2017b, LiD2016a} for the local Cahn--Hilliard model
could hardly be extended to the case of nonlocal models due to the lack of higher order diffusion terms.
Instead, we view the numerical solution as a perturbation of the exact solution to \eqref{nonlocalCH},
perform a local in time convergence analysis,
and obtain the $\ell^\infty$ bound of the numerical solution via the convergence result.
All the analysis will be specified in the 2-D case,
similar results can be obtained for the 1-D and 3-D cases without any extra essential difficulties.

The outline of the paper is as follows.
Some notations and lemmas for the spectral collocation method for the spatial discretization are summarized in Section 2.
The convergence analysis, as well as the $\ell^\infty$ bound of the numerical solutions,
of the first order stabilized linear semi-implicit scheme \eqref{scheme-1st-1} is presented in Section 3.
Finally, some concluding remarks are given in Section 4.

\section{Spectral collocation method for the spatial discretization}


In this section, we summarize some notations and lemmas introduced in \cite{du18}
for the spectral collocation approximations
of some spatial operators in the two-dimensional space with $\Omega=(-X,X)\times(-Y,Y)$.

Let $N_x$ and $N_y$ be two even numbers.
The $N_x\times N_y$ mesh $\Omega_h$ of the domain $\Omega$ is
a set of nodes $(x_i,y_j)$ with $x_i=-X+ih_x$, $y_j=-Y+jh_y$, $1\le i\le N_x$, $1\le j\le N_y$,
where $h_x=2X/N_x$ and $h_y=2Y/N_y$ are the uniform mesh sizes in each direction.
Let $h=\max\{h_x,h_y\}$.
We define the index sets
\begin{align*}
S_h & =\{(i,j)\in\Z^2\,|\,1\le i\le N_x,\ 1\le j\le N_y\},\\
\widehat{S}_h & =\Big\{(k,l)\in\Z^2\,\Big|\,-\frac{N_x}{2}+1\le k\le\frac{N_x}{2},\ -\frac{N_y}{2}+1\le l\le\frac{N_y}{2}\Big\}.
\end{align*}
All of the periodic grid functions defined on $\Omega_h$ are denoted by $\Mh$, that is,
$$\Mh=\{f:\Omega_h\to\R\,|\,\text{$f_{i+mN_x,j+nN_y}=f_{ij}$ for any $(i,j)\in S_h$ and $(m,n)\in\Z^2$}\}.$$
For any $f,g\in\Mh$ and $\bm{f}=(f^1,f^2)^T,\bm{g}=(g^1,g^2)^T\in\Mh\times\Mh$,
the discrete $L^2$ inner product $\<\cdot,\cdot\>$, discrete $L^2$ norm $\|\cdot\|_2$,
and discrete $L^\infty$ norm $\|\cdot\|_\infty$ are respectively defined  by
\[
\begin{array}{ll}
\ds\<f,g\>=h_xh_y\sum_{(i,j)\in {S}_h}f_{ij}g_{ij}, & \quad\ds\<\bm{f},\bm{g}\>=h_xh_y\sum_{(i,j)\in {S}_h}(f^1_{ij}g^1_{ij}+f^2_{ij}g^2_{ij}),\\
\ds\|f\|_2=\sqrt{\<f,f\>}, & \quad\ds\|\bm{f}\|_2=\sqrt{\<\bm{f},\bm{f}\>},\\
\ds\|f\|_\infty=\max_{(i,j)\in S_h}|f_{ij}|,
& \quad\ds\|\bm{f}\|_\infty=\max_{(i,j)\in S_h}\sqrt{| f^1_{ij} |^2 + | f^2_{ij} |^2}.
\end{array}
\]
For any $f\in\Mh$, we call $\overline{f}:=\frac{1}{4XY}\<f,1\>$ the mean value of $f$.
In particular, denote by $\Mh^0$ all the grid functions in $\Mh$ with mean zero, i.e.,
$$\Mh^0=\{f\in\Mh\,|\,\<f,1\>=0\}.$$

\subsection{Discrete gradient, divergence and Laplace operators}

For a function $f\in\Mh$, the 2-D discrete Fourier transform $\hat{f}=Pf$ is defined componentwisely  \cite{ShTaWa11,Tre00} by
\begin{equation}
\label{eq_distransform}
\hat{f}_{kl}=\sum_{(i,j)\in S_h}f_{ij}\exp\Big(-\i\frac{k\pi}{X}x_i\Big)\exp\Big(-\i\frac{l\pi}{Y}y_j\Big),\qquad(k,l)\in\widehat{S}_h.
\end{equation}
The function $f$ can be reconstructed via the corresponding inverse transform $f=P^{-1}\hat{f}$ with components given by
\begin{equation}
\label{eq_disinvtransform}
f_{ij}=\frac{1}{N_xN_y}\sum_{(k,l)\in\widehat{S}_h}\hat{f}_{kl}\exp\Big(\i\frac{k\pi}{X}x_i\Big)\exp\Big(\i\frac{l\pi}{Y}y_j\Big),
\qquad(i,j)\in {S}_h.
\end{equation}

Let $\widehat{\mathcal{M}}_h=\{Pf\,|\,f\in\Mh\}$
and define the operators $\widehat{D}_x$ and $\widehat{D}_y$ on $\widehat{\mathcal{M}}_h$ as
$$(\widehat{D}_x\hat{f})_{kl}=\Big(\frac{k\pi\i}{X}\Big)\hat{f}_{kl},\quad
(\widehat{D}_y\hat{f})_{kl}=\Big(\frac{l\pi\i}{Y}\Big)\hat{f}_{kl},\quad(k,l)\in\widehat{S}_h,$$
then the Fourier spectral approximations to the first and second partial derivatives can be represented as
$$D_x=P^{-1}\widehat{D}_xP,\qquad D_y=P^{-1}\widehat{D}_yP,\qquad
D_x^2=P^{-1}\widehat{D}_x^2P,\qquad D_y^2=P^{-1}\widehat{D}_y^2P.$$
For any $f\in\Mh$ and $\bm{f}=(f^1,f^2)^T\in\Mh\times\Mh$,
the discrete gradient, divergence and Laplace operators are given respectively  by
$$\nabla_Nf=\binom{D_xf}{D_yf},\qquad\nabla_N\cdot\bm{f}=D_xf^1+D_yf^2,\qquad
\Delta_Nf=D_x^2f+D_y^2f.$$
It is easy to prove the following results.

\begin{lemma}
\label{lem_disoperator}
{\rm(i)} For any $f,g\in\Mh$ and $\bm{g}\in\Mh\times\Mh$,
we have the summation by parts formulas
$$\<f,\nabla_N\cdot\bm{g}\>=-\<\nabla_Nf,\bm{g}\>,
\qquad\<f,\Delta_Ng\>=-\<\nabla_Nf,\nabla_Ng\>=\<\Delta_Nf,g\>.$$
{\rm(ii)} The inversion of $-\Delta_N$ exists on $\Mh^0$
and $(-\Delta_N)^{-1}$ is self-adjoint and positive definite.
\end{lemma}

Lemma \ref{lem_disoperator} (ii) tells us that $(-\Delta_N)^{-1}f$ is well-defined for any $f\in\Mh^0$.
Then we can define the discrete $H^{-1}$ norm $\nrm{\,\cdot\,}_{-1,N}$ by
\begin{equation}
\nrm{f}_{-1,N} = \sqrt{\< f, (-\Delta_N)^{-1}f\>}, \quad \forall \ f\in\Mh^0.
\end{equation}

\subsection{Discrete convolution and nonlocal operators}

To define the discrete convolutions, we consider the kernel function set
$$\mathcal{K}_h=\{\psi:\Omega_{h,0}\to\R\,|\,\text{$\psi_{i+mN_x,j+nN_y}=\psi_{ij}$ for any $(i,j)\in S_h$ and $(m,n)\in\Z^2$}\},$$
where $\Omega_{h,0}=\{(ih_x,jh_y)\,|\,(i,j)\in S_h\}$ is the mesh on the domain $(0,2X)\times(0,2Y)$.
A discrete transform and its inversion of a function $\psi\in\mathcal{K}_h$ could be defined similarly via \eqref{eq_distransform} and \eqref{eq_disinvtransform}
by replacing $x_i$ and $y_j$ by $ih_x$ and $jh_y$, respectively.
Actually, $\mathcal{K}_h$ is equivalent to $\Mh$ due to the periodicity of their elements,
and we consider the functions from $\mathcal{K}_h$ as the kernels just for convenience of notations.

For any $\psi\in\mathcal{K}_h$ and $f\in\Mh$, the discrete convolution $\psi\os f\in\Mh$ is defined componentwisely by
$$(\psi\os f)_{ij}=h_xh_y\sum_{(m,n)\in S_h}\psi_{i-m,j-n}f_{mn},\qquad(i,j)\in S_h.$$
Especially, by setting $f\equiv1$ on $\Omega_h$, we have
$$\psi\os1=h_xh_y\sum_{(m,n)\in S_h}\psi_{mn}.$$
%
The following preliminary estimate is needed in the convergence analysis.

\begin{lemma}
\label{lem:1}
Suppose $\mathsf{J} \in \Cp^1(\Omega)$ and define its grid restriction by $J_{ij}:=\mathsf{J}(x_i,y_j)$.
Then for any $f,g\in\Mh$, we have
\begin{equation}
\label{lem 1:0}
| \< J \os f, \Delta_N g \> | \le \alpha \| f \|^2_2 + \frac{C}{\alpha} \| \nabla_N g \|^2_2 ,
\end{equation}
for any $\alpha > 0$,
where $C$ is a positive constant that depends on $\mathsf{J}$ but is independent of $h_x$ and $h_y$.
\end{lemma}

\begin{proof}
An application of summation by parts and Cauchy--Schwarz inequality shows that
\begin{equation}
| \< J \os f, \Delta_N g \> | = | \< \nabla_N (J \os f), \nabla_N g \> | \le \| \nabla_N (J \os f)\|_2\cdot\|\nabla_N g\|_2.
\end{equation}
An application of the definitions of discrete gradient and convolution gives us
\begin{equation}
( \nabla_N (J \os f) )_{ij} = h_xh_y\sum_{(m,n)\in S_h}(\nabla_N J)_{i-m,j-n}f_{mn},\qquad(i,j)\in S_h.  \label{lemma 2.2-1}
\end{equation}
Then,
\begin{align*}
\| \nabla_N (J \os f) \|_2^2
& = h_xh_y\sum_{(i,j)\in S_h} \Big(h_xh_y\sum_{(m,n)\in S_h}(\nabla_N J)_{i-m,j-n}f_{mn}\Big)^2 \\
& \le |\Omega|\|\nabla_N J\|_\infty^2 \Big(h_xh_y\sum_{(m,n)\in S_h}f_{mn}\Big)^2 \\
& \le |\Omega|^2\|\nabla_N J\|_\infty^2 \|f\|_2^2.
\end{align*}
The smoothness of $\mathsf{J}$ implies the bound $\|\nabla_NJ\|_\infty\le C_{\mathsf{J}}$.
Then, we arrive
\begin{equation}
|\<J\os f, \Delta_Ng\>|\le C_{\mathsf{J}}|\Omega|\|f\|_2\|\nabla_Ng\|_2
\le\alpha\|f\|_2^2 + \frac{C}{\alpha}\|\nabla_Ng\|_2^2,
\end{equation}
for any $\alpha>0$, where $C=\frac{1}{4}C_{\mathsf{J}}^2|\Omega|^2$.
\end{proof}

\begin{remark}
We make a technical assumption $\mathsf{J} \in \Cp^1(\Omega)$ to facilitate the analysis. On the other hand, a singular $\mathsf{J}$ may also be of certain scientific interests in many relevant physical models. However, a direct application of the above analysis is not available for a singular $\mathsf{J}$, since a point-wise bound of $\nabla_N J$ is not available in~\eqref{lemma 2.2-1} any more. A non-standard extension to the case with a singular $\mathsf{J}$ will be considered in our future works.
\end{remark}

Given an integrable kernel $J$ satisfying the assumptions (a)--(d),
we can define the discrete version of the nonlocal operator $\hL$ by
\begin{equation}
\hL_Nf=(J\os1)f-J\os f,\qquad\forall\,f\in\Mh.
\end{equation}
It is easy to check that $\hL_N$ commutes with $\Delta_N$ and is self-adjoint and positive semi-definite.
Finally, the discrete version of the energy \eqref{nonlocalenergy2} is defined as
\begin{equation}
E_N(v)=\<F(v),1\>+\frac{\eps^2}{2}\<\hL_Nv,v\>,\quad v\in\Mh.
\end{equation}

\subsection{Fourier projection of the exact solution}

The existence and uniqueness of a smooth periodic solution to the IPDE \eqref{nonlocalCH} with smooth periodic initial data
may be established by using techniques developed by Bates and Han in \cite{bates05b, bates05a}.
In this article, we denote this IPDE solution by $U$.
Motivated by these results, one obtains
\begin{equation}
\nrm{U}_{L^\infty(0,T;L^\infty)} + \nrm{ \nabla U}_{L^\infty(0,T;L^\infty)} < C  ,
\end{equation}
for any $T>0$.

Define $U_N (\, \cdot \, ,t) := {\cal P}_N U (\, \cdot \, ,t)$,
the (spatial) Fourier projection of the exact solution into ${\cal B}^N$,
the space of trigonometric polynomials of degree up to $N$.
The following projection approximation is standard:
if $U \in L^\infty(0,T;\Hp^\ell)$, for some $\ell\in\mathbb{N}$,
\begin{equation}
\| U_N - U \|_{L^\infty(0,T;H^k)}
\le C h^{\ell-k} \| U \|_{L^\infty(0,T;H^\ell)},  \, \, \,
h = \max \{ h_x, h_y \} ,  \quad \forall \ 0 \le k \le \ell .
\end{equation}
By $U_N^m$,  $U^m$ we denote $U_N(\, \cdot \, , t_m)$ and $U(\, \cdot \, , t_m)$, respectively, with $t_m = m\dt$.
It is clear that $\int_\Omega \, U_N ( \cdot, t_m) \, d {\bf x} = \int_\Omega \, U ( \cdot, t_m) \, d {\bf x}$, for any $m \in\mathbb{N}$, due to the fact that $U_N$ is the Fourier projection of $U$, and thus,
\begin{align}
\int_\Omega \, U_N ( \cdot, t_m) \, d {\bf x}
= \int_\Omega \, U ( \cdot, t_m) \, d {\bf x}
= \int_\Omega \, U ( \cdot, t_{m-1}) \, d {\bf x}
  =  \int_\Omega \, U_N ( \cdot, t_{m-1}) \, d {\bf x},
\quad \forall \ m \in\mathbb{N} ,
\label{mass conserv-1}
\end{align}
in which the second step is based on the fact that the exact solution $U$ is mass conservative at the continuous level.
On the other hand, the solution of the numerical scheme~\eqref{scheme-1st-1} is also mass conservative at the discrete level:
\begin{equation}
\overline{u^m} = \overline{u^{m-1}} ,  \quad \forall \ m \in \mathbb{N} .
\end{equation}
Meanwhile, we denote by $u_N^m$ the values of $U_N$ at discrete grid points at time instant $t_m$, i.e.,
$u_N^m := {\mathcal P}_h U_N (\, \cdot \, , t_m)$.
Since $U_N \in {\cal B}^N$, it always holds
\[
\int_\Omega U_N(\, \cdot \, ,t_m)  \, d {\bf x}
= h_x h_y \sum_{(i,j)\in S_h} U_N(x_i,y_j,t_m)
= h_x h_y \sum_{(i,j)\in S_h} (u_N^m)_{i,j},
\]
so the mass conservative property is available at the discrete level: $\overline{u_N^m} = \overline{u_N^{m-1}}$.
As indicated before, we use the mass conservative projection for the initial data:
$u^0 = u_N^0 = {\mathcal P}_h U_N (\, \cdot \, , t=0)$, that is
\begin{equation}
u^0_{i,j} := U_N (x_i, y_j, t=0) ,
\label{initial data-0}
\end{equation}	
The error grid function is defined as
\begin{equation}
e^m := u_N^m - u^m ,  \quad \forall \ m \ge 0 .
\end{equation}
Therefore, it follows that
\begin{equation}
  \overline{e^m} =0 ,  \quad
  \mbox{since $\overline{u_N^m} = \overline{u_N^0} = \overline{u^0} = \overline{u^m}$},  \, \, \, \forall \, m \ge 0 ,
   \label{error function-zero mean-1}
\end{equation}
so that the discrete norm $\nrm{ \, \cdot \, }_{-1,N}$ is well defined for the error grid function. We also notice that the Fourier projection of the exact solution has to be taken at the initial time step as~\eqref{initial data-0}, instead of a pointwise interpolation of the exact initial data, to assure the zero-mean property of the numerical error grid function at a discrete level. In addition, we have (with $N_k:= [ \frac{T}{\dt} ]$ denoting the integer part of $\frac{T}{\dt}$),
\begin{equation}
\max_{1\le k \le N_k} \| u_N^k\|_\infty + \max_{1\le k \le N_k}\| \nabla_N u_N^k\|_\infty < C .
\end{equation}

\section{Convergence analysis and energy stability analysis}

We begin this section by stating the main result on the convergence analysis of the stabilized linear scheme \eqref{scheme-1st-1}.
The detailed proof is given in the following two subsections,
including the higher order consistency analysis and the convergence analysis.
The energy stability of \eqref{scheme-1st-1} is then obtained
under some new assumptions on $A$, instead of \eqref{condition-A-0} given in \cite{du18}. With an initial data with sufficient regularity, we could assume that the exact solution has regularity of class $\mathcal{R}$:
\begin{equation}
U \in \mathcal{R} := H^4 (0,T; C_{\rm per}^0) \cap H^3 (0,T; C_{\rm per}^2) \cap L^\infty (0,T; C_{\rm per}^{m+2})  ,  \quad m \ge 3 . 			 
\end{equation}

\begin{theorem}
\label{thm:convergence-Hm1}
Given periodic initial data $U(x,y,t=0) \in C_{\rm per}^{m+2}(\Omega)$.
Suppose the unique periodic solution for the IPDE \eqref{nonlocalCH},
given by $U(x,y,t)$ on $\Omega\times(0,T]$ for some $T< \infty$, is of regularity class $\mathcal{R}$. In addition, the following assumption is made for the constant $A$:
\begin{equation}
\label{condition-A-1}
A \ge \frac{18 M_0^4}{\gamma_0} ,  \quad \text{with} \
M_0 =  1 + \max_{1\le k \le N_k} \| u_N^k \|_\infty .
\end{equation}
Then, provided $\dt$ and $h$ are sufficiently small, under linear refinement path constraint $\dt \le C h$, with $C$ any fixed constant, we have
\begin{equation}
\label{convergence-0}
\| e^n \|_{-1, N} + \Bigl( \gamma_0 \dt \sum_{k=1}^n \| e^k \|_2^2  \Bigr)^{\frac12} \le C (\dt + h^m ) ,
\end{equation}
for all positive integers $n$, such that $n \dt \le T$, where $C>0$ is independent of $h$ and $\dt$.
\end{theorem}

\subsection{Higher order consistency analysis}


By consistency, the Fourier projection solution $U_N$ solves the discrete equation with an $O(\dt + h^m)$ accuracy:
\begin{equation}
\frac{U_N^{n+1} - U_N^n}{\dt} = \Delta_N \Bigl( (U_N^n)^3 - U_N^n + A (U_N^{n+1} - U_N^n) + \eps^2 {\cal L}_N U_N^{n+1} \Bigr) + \tau_0^{n+1}  ,
\end{equation}
where the local truncation error $\tau_0^{n+1}$ satisfies
\begin{equation}
\| \tau_0^{n+1} \|_{-1,N} \le C (\dt + h^m)  .   \label{truncation-error-2}
\end{equation}
In addition, the discrete zero-mean property of $\tau_0^{n+1}$ is observed, which will be useful in later analysis:
\begin{equation}
  \overline{\tau_0^{n+1}} = 0 ,  \quad \mbox{since} \, \, \,
  \overline{U_N^{n+1}} = \overline{U_N^n},  \, \,
  \int_\Omega \, \Delta_N \Bigl( (U_N^n)^3 - U_N^n + A (U_N^{n+1} - U_N^n) + \eps^2 {\cal L}_N U_N^{n+1} \Bigr) \, d \bx = 0 .   \label{truncation-error-zero mean}
\end{equation}
Notice that the first identity is derived in~\eqref{mass conserv-1},
while the second one comes from the periodic boundary condition.
However, this local truncation error will not be enough
to recover the $\| \cdot \|_\infty$ bound of the numerical solution due to the first order accuracy in time.
To remedy this, we have to construct supplementary fields, $U^1_{\dt}$, $U^2_{\dt}$, and denote
\begin{equation}
\label{consistency-1}
\hat{U} = U_N + \dt {\cal P}_N U^1_{\dt} + \dt^2 {\cal P}_N U^2_{\dt}  .
\end{equation}
We also notice that both $U^1_{\dt}$, $U^2_{\dt}$ are (spatially) continuous functions, and their construction will be outlined later.
Moreover, a higher $O (\dt^3 + h^m)$ consistency has to be satisfied with the given numerical scheme \eqref{scheme-1st-1}.
The constructed fields $U^1_{\dt}$, $U^2_{\dt}$, which will be obtained using a perturbation expansion, will depend solely on the exact solution $U$.

In other words, we introduce a higher order approximate expansion  of the exact solution, since a first order temporal consistency estimate~\eqref{truncation-error-2} is not able to control the discrete $\ell^\infty$ norm of the numerical solution.
Instead of substituting the exact solution into the numerical scheme, a careful construction of an approximate profile is performed by adding $O (\dt)$ and $O (\dt^2)$ correction terms to the exact solution to satisfy an $O (\dt^3)$ truncation error. In turn, we estimate the numerical error function between the constructed profile and the numerical solution, instead of a direct comparison between the numerical solution and exact solution. Such a higher order consistency enables us to derive a higher order convergence estimate in the $\| \cdot \|_{-1,N}$ norm, which in turn leads to a desired $\| \cdot \|_\infty$ bound of the numerical solution, via an application of inverse inequality.   This approach has been reported for a wide class of nonlinear PDEs; see the related works for the incompressible fluid equation~\cite{E95, E02, STWW03, STWW07, WL00, WL02, WLJ04}, various gradient equations~\cite{baskaran13b, guan17a, guan14a}, the porous medium equation based on the energetic variational approach~\cite{duan19c}, nonlinear wave equation~\cite{WangL15}, etc.

We begin with an application of the temporal discretization in the numerical scheme \eqref{scheme-1st-1} for the Fourier projection solution $U_N$:
\begin{equation}
\label{consistency-2-1}
\frac{U_N^{n+1} - U_N^n}{\dt}  = \Delta \Bigl( (U_N^n)^3 - U_N^n + A (U_N^{n+1} - U_N^n)
+ \eps^2 {\cal L} U_N^{n+1} \Bigr) + \dt \g^{(1)} + O (\dt^2) + O (h^m) ,
\end{equation}
which comes from the Taylor expansion in time. In more details, the function $\g^{(1)}$ is smooth enough and only depends on the higher order derivatives of $U_N$. In particular, by making use of similar arguments as in the derivation of ~\eqref{truncation-error-zero mean}, we conclude that
\begin{equation}
   \int_\Omega \, ( U_N^{n+1} - U_N^n ) \, d \bx = 0 ,  \, \, \,
    \int_\Omega \, \Delta \Bigl( (U_N^n)^3 - U_N^n + A (U_N^{n+1} - U_N^n)
+ \eps^2 {\cal L} U_N^{n+1} \Bigr)  \, d \bx = 0 .
   \label{consistency-2-1-2}
\end{equation}
This in turn indicates that
\begin{equation}
   \int_\Omega \, \g^{(1)} \, d \bx = 0 .
   \label{consistency-2-1-3}
\end{equation}
The first order temporal correction function $U^1_{\dt}$ is given by the solution of the following linear differential equation
\begin{align}
    &
\partial_t U^1_{\dt}  = \Delta  \Bigl(  3 U_N^2 U^1_{\dt} - U^1_{\dt}
 + \eps^2 {\cal L} U^1_{\dt} \Bigr) - \g^{(1)}  ,   \label{consistency-2-2}
\\
  &
  U^1_{\dt} ( \cdot , \, t=0) \equiv 0  ,  \label{consistency-2-2-b}
\end{align}
with the periodic boundary condition.
In fact, \eqref{consistency-2-2} is a linear parabolic PDE, with a sufficiently regular coefficient function $3 U_N^2$ (regularity dependent on $U \in {\cal R}$). The existence and uniqueness of its solution could be derived by making use of a standard Galerkin procedure and Sobolev estimates, following the classical techniques for time-dependent parabolic equation~\cite{temam01}. Such a solution depends solely on the profile $U_N$ and is regular enough.
Similar to \eqref{consistency-2-1}, an application of the temporal discretization to $U^1_{\dt}$ indicates that
\begin{align}
\frac{ (U^1_{\dt})^{n+1} - ( U^1_{\dt} )^n}{\dt}
& = \Delta \Bigl( 3 ( U_N^n )^2 (U^1_{\dt})^n  - (U^1_{\dt})^n + A ( (U^1_{\dt})^{n+1} - (U^1_{\dt})^n) \nonumber\\
& \qquad   + \eps^2 {\cal L} ( U^1_{\dt} )^{n+1}  \Bigr)    - ( \g^{(1)} )^n + O (\dt) . \label{consistency-2-3}
\end{align}
In turn, we denote $\hat{U}^{(1)} = U_N + \dt {\cal P}_N U^1_{\dt}$. 
A combination of \eqref{consistency-2-1} and a Fourier projection of \eqref{consistency-2-3} results in the following higher order consistency estimate:
\begin{align}
\frac{ (\hat{U}^{(1)} )^{n+1} - ( \hat{U}^{(1)} )^n}{\dt}
& = \Delta \Bigl( ( ( \hat{U}^{(1)} )^n )^3 - (\hat{U}^{(1)} )^n + A ( (\hat{U}^{(1)} )^{n+1}  - (\hat{U}^{(1)} )^n)  \nonumber\\
& \qquad   + \eps^2 {\cal L} ( \hat{U}^{(1)} )^{n+1}  \Bigr) + \dt^2 \g^{(2)} + O (\dt^3) + O (h^m), \label{consistency-2-4}
\end{align}
in which we have made use of the following estimate
\begin{align}
  (\hat{U}^{(1)} )^3 = ( U_N + \dt {\cal P}_N U^1_{\dt} )^3
  &= U_N^3  + 3 \dt U_N^2 {\cal P}_N U^1_{\dt} + O (\dt^2)  \nonumber
\\
  &=
  U_N^3  + 3 \dt {\cal P}_N ( U_N^2 {\cal P}_N U^1_{\dt} ) + O (\dt^2)
  + O (h^m) .
\end{align}
Again, $\g^{(2)}$ is smooth enough and only dependent on the higher order derivatives of $U_N$.

In addition, we observe that the constructed profile $U^1_{\dt}$ has zero-mean at the continuous level, based on the equation~\eqref{consistency-2-2}-\eqref{consistency-2-2-b}, combined with the fact~\eqref{consistency-2-1-3}:
\begin{equation}
\int_\Omega \, \partial_t U^1_{\dt}  \, d \bx = \int_\Omega \, \Delta  \Bigl(  3 U_N^2 U^1_{\dt} - U^1_{\dt}+ \eps^2 {\cal L} U^1_{\dt} \Bigr) \, d \bx
 - \int_\Omega \, \g^{(1)} \, d \bx  = 0 ,   \label{consistency-2-5-b}
\\
\end{equation}
so that
\begin{equation}
  \int_\Omega \, U^1_{\dt} ( \cdot , \, t) \, d \bx
  = \int_\Omega \, U^1_{\dt} ( \cdot , \, t=0) \, d \bx
  = 0 ,  \, \, \, \forall \, t > 0 .  
\end{equation}
In turn, its projection also has a zero-mean:
\begin{equation}
  \int_\Omega \, {\cal P}_N U^1_{\dt} ( \cdot , \, t) \, d \bx
  = 0 ,  \, \, \, \forall \, t > 0 .  
\end{equation}
Therefore, we conclude that $\hat{U}^{(1)}$ has the same average as $U_N$ at the continuous level:
\begin{equation}
  \int_\Omega \, \hat{U}^{(1)} ( \cdot , \, t) \, d \bx
  = \int_\Omega \, U_N ( \cdot , \, t) \, d \bx ,
  \, \, \, \forall \, t > 0 .   
\end{equation}
Since $U_N$ is mass conservative at the continuous level, as indicated by~\eqref{mass conserv-1}, we arrive at a similar property for $\hat{U}^{(1)}$:
\begin{equation}
  \int_\Omega \, (\hat{U}^{(1)} )^{n+1} \, d \bx
  = \int_\Omega \, (\hat{U}^{(1)} )^n \, d \bx ,
  \, \, \, \forall  n \in\mathbb{N}  .    \label{consistency-2-5-f}
\end{equation}
As a consequence, by making use of similar arguments as in~\eqref{consistency-2-1-2}-\eqref{consistency-2-1-3}, we see that $\g^{(2)}$ has zero-mean at the continuous level:
\begin{equation}
   \int_\Omega \, \g^{(2)} \, d \bx = 0 .
\end{equation}

The second order temporal correction function $U^2_{\dt}$ could be constructed in a similar manner, and it turns out to be the solution of the following linear differential equation
\begin{align}
  &
\partial_t U^2_{\dt} = \Delta  \Bigl(  3 U_N^2 U^2_{\dt} - U^2_{\dt} + \eps^2 {\cal L} U^2_{\dt} \Bigr) - \g^{(2)}  ,  \label{consistency-2-6}
\\
  &
  U^2_{\dt} ( \cdot , \, t=0) \equiv 0  , 
\end{align}
with the periodic boundary condition.
Similarly, \eqref{consistency-2-6} is a linear parabolic PDE, with a sufficiently regular coefficient function $3 U_N^2$, and its unique solution depends solely on the profile $U_N$ and is smooth enough.
An application of the temporal discretization to $U^2_{\dt}$ gives
\begin{align}
\frac{ (U^2_{\dt})^{n+1} - ( U^2_{\dt} )^n}{\dt}
& = \Delta \Bigl( 3 (U_N^n)^2 (U^2_{\dt})^n - (U^2_{\dt})^n + A ( (U^2_{\dt})^{n+1} - (U^2_{\dt})^n)  \nonumber\\
& \qquad  + \eps^2 {\cal L} ( U^2_{\dt} )^{n+1}  \Bigr) - ( \g^{(2)} )^n + O (\dt) . \label{consistency-2-7}
\end{align}
Notice that $\hat{U} = \hat{U}^{(1)} + \dt^2 {\cal P}_N U^2_{\dt}$.
In turn, a combination of \eqref{consistency-2-4} and \eqref{consistency-2-7}
leads to the desired third order consistency estimate in time:
\begin{equation}
\frac{ \hat{U}^{n+1} - \hat{U}^n}{\dt}
 = \Delta \Bigl( ( \hat{U}^n )^3 - \hat{U}^n + A ( \hat{U}^{n+1} - \hat{U}^n) + \eps^2 {\cal L} \hat{U}^{n+1} \Bigr)
  + O (\dt^3) + O (h^m),
\end{equation}
in which we have made use of the following estimate
\begin{align}
\hat{U}^3 = ( \hat{U}^{(1)} + \dt^2 {\cal P}_N U^2_{\dt} )^3
 &= (\hat{U}^{(1)} )^3  + 3 \dt^2 U_N^2 {\cal P}_N U^2_{\dt}  + O (\dt^3)   \nonumber
\\
  &=
  (\hat{U}^{(1)} )^3  + 3 \dt^2 {\cal P}_N ( U_N^2 {\cal P}_N U^2_{\dt} )
 + O (\dt^3) + O (h^m) .
\end{align}

Similar to the analyses in~\eqref{consistency-2-5-b}-\eqref{consistency-2-5-f},
we are able to prove that $\hat{U}$ has the same average as $U_N$ at the continuous level:
\begin{equation}
\int_\Omega \, \partial_t U^2_{\dt}  \, d \bx = \int_\Omega \, \Delta  \Bigl(  3 U_N^2 U^2_{\dt} - U^2_{\dt}+ \eps^2 {\cal L} U^2_{\dt} \Bigr) \, d \bx
 - \int_\Omega \, \g^{(2)} \, d \bx  = 0 , 
\end{equation}
so that
\begin{eqnarray}
  &&
  \int_\Omega \, U^2_{\dt} ( \cdot , \, t) \, d \bx
  = \int_\Omega \, U^2_{\dt} ( \cdot , \, t=0) \, d \bx
  = 0 ,  \, \, \, \forall \, t > 0 ,  
\\
  &&
  \int_\Omega \, {\cal P}_N U^2_{\dt} ( \cdot , \, t) \, d \bx
  = 0 ,  \, \, \, \forall \, t > 0 .  
\\
  &&
    \int_\Omega \, \hat{U} ( \cdot , \, t) \, d \bx
  = \int_\Omega \, \hat{U}^{(1)} ( \cdot , \, t) \, d \bx
  = \int_\Omega \, U_N ( \cdot , \, t) \, d \bx ,
  \, \, \, \forall \, t > 0 .    \label{consistency-2-9-e}
\\
  &&
  \int_\Omega \, \hat{U}^{n+1} \, d \bx
  = \int_\Omega \, \hat{U}^n \, d \bx ,
  \, \, \, \forall  n \in\mathbb{N}  .    \label{consistency-2-9-f}
\end{eqnarray}

Finally, with an application of Fourier pseudo-spectral approximation in space,
we obtain the $O (\dt^3 + h^m)$ truncation error estimate for the constructed solution $\hat{U}$:
\begin{eqnarray}
  &&
\frac{ \hat{U}^{n+1} - \hat{U}^n}{\dt}
= \Delta_N \Bigl( ( \hat{U}^n )^3 - \hat{U}^n + A ( \hat{U}^{n+1} - \hat{U}^n) + \eps^2 {\cal L}_N \hat{U}^{n+1} \Bigr) + \tau_2^{n+1} ,  \label{consistency-3-1}
\\
  &&
  \mbox{with} \quad \| \tau_2^{n+1} \|_{-1,N} \le C (\dt^3 + h^m)  .
\end{eqnarray}
We notice that $\tau_2^{n+1}$ has zero-mean at a discrete level, $\overline{\tau_2^{n+1} } =0$, for any $n \in\mathbb{N}$, based on the estimate~\eqref{consistency-2-9-f}, combined with the fact that $\hat{U}^k \in {\cal B}^N$.
	
As stated earlier, the purpose of the higher order expansion \eqref{consistency-1} is
to obtain an $\ell^{\infty}$ bound of the error function via its $\| \cdot \|_{-1,N}$ norm in higher order accuracy
by utilizing an inverse inequality in spatial discretization, which will be shown below.
A detailed analysis shows that
\begin{equation}
\|  \hat{U} - U_N \|_\infty \le C  \dt   ,
\end{equation}
since $\| U^1_{\dt} \|_\infty, \| U^2_{\dt} \|_\infty \le C$.
In particular, the following bound becomes available:
\begin{equation}
\label{consistency-4-2}
\| \hat{U} - U_N \|_\infty \le C  \dt   \le \frac12 ,
\end{equation}
provided that $\dt$ is sufficiently small, so that
\begin{equation}
\label{consistency-4-3}
\| \hat{U} \|_\infty \le \| U_N \|_\infty + \| \hat{U} - U_N \|_\infty
\le \| U_N \|_\infty + \frac12 .
\end{equation}

\subsection{Convergence analysis in the $\ell^\infty (0,T; H_h^{-1}) \cap \ell^2 (0,T; \ell^2 )$ norm}

Instead of a direct comparison between the numerical solution and the Fourier projection $U_N$ of the exact solution,
we estimate the error between the numerical solution and the constructed solution
to obtain a higher order convergence in $\| \cdot \|_{-1, N}$ norm.
In turn, the following error function is introduced:
\begin{equation}
\hat{e}^k := \hat{U}^k - u^k .
\end{equation}
In particular, the established consistency estimate~\eqref{consistency-2-9-e} indicates that
\begin{equation}
   \overline{\hat{U}^k} = \frac{1}{| \Omega |} \int_\Omega \, \hat{U}^k \, d \bx
   = \frac{1}{| \Omega |} \int_\Omega \, U_N^k \, d \bx  ,
   \quad \forall   k \in\mathbb{N} ,  
\end{equation}
in which the first step is based on the fact that $\hat{U}^k \in {\cal B}^N$. Its combination with~\eqref{error function-zero mean-1} results in the discrete zero-mean property of the numerical error function $\hat{e}^k$:
\begin{equation}
  \overline{\hat{e}^k} =0 ,  \quad \mbox{since $\overline{\hat{U}^k} = \overline{U_N^k} = \overline{U_N^0} = \overline{u^0} = \overline{u^k}$},  \, \, \, \forall \, k \ge 0 .
\end{equation}
In turn, the discrete $\| \cdot \|_{-1,N}$ norm of this error function is well defined.

Subtracting~\eqref{scheme-1st-1} from~\eqref{consistency-3-1} yields
\begin{eqnarray}
\label{convergence-1}
 \frac{ \hat{e}^{n+1} - \hat{e}^n}{\dt}
 = \Delta_N \Bigl(  ( \hat{U}^n )^3 - (u^n)^3 - \hat{e}^n + A ( \hat{e}^{n+1} - \hat{e}^n)  + \varepsilon^2 {\cal L}_N \hat{e}^{n+1} \Bigr) + \tau_2^{n+1} .
\end{eqnarray}
To carry out the nonlinear error estimate,
we make an $\| \cdot \|_\infty$ assumption for the numerical error function at the previous time step $t_n$:
\begin{equation}
\label{a priori-1}
\| \hat{e}^n \|_\infty \le \frac12  .
\end{equation}
In turn, the $\| \cdot \|_\infty$ bound for the numerical solution at $t^n$ becomes available
\begin{equation}
\label{a priori-2}
  \| u^n \|_\infty = \| \hat{U}^n - \hat{e}^n \|_\infty
  \le \| \hat{U}^n \|_\infty + \| \hat{e}^n \|_\infty
  \le \| U_N^n \|_\infty + \frac12 + \frac12 \le M_0 ,
\end{equation}
in which the estimate \eqref{consistency-4-2} for $\| \hat{U}^n \|_\infty$ has been recalled in the third step.
The \emph{a priori} assumption \eqref{a priori-1} will be recovered in the convergence estimate at the next time step,
as will be demonstrated later.

Since $\overline{\hat{e}^k}=0$ for any $k \ge 0$, $(-\Delta_N)^{-1} \hat{e}^k$ has been well-defined.
Taking a discrete inner product with \eqref{convergence-1} by $2 (-\Delta_N)^{-1} \hat{e}^{n+1}$ leads to
\begin{align}
& \quad~
 \| \hat{e}^{n+1} \|_{-1,N}^2 - \| \hat{e}^n \|_{-1,N}^2
 +  \| \hat{e}^{n+1} - \hat{e}^n \|_{-1,N}^2
 + 2A \dt \langle \hat{e}^{n+1} - \hat{e}^n, \hat{e}^{n+1} \rangle  \nonumber
\\
  &=
    - 2 \dt \langle ( \hat{U}^n )^3 - (u^n)^3 , \hat{e}^{n+1} \rangle
  + 2 \dt \langle \hat{e}^n , \hat{e}^{n+1} \rangle
   - 2 \varepsilon^2 \dt \langle {\cal L}_N \hat{e}^{n+1} , \hat{e}^{n+1} \rangle  \nonumber
 \\
  &\qquad
  + 2 \dt \langle (-\Delta_N)^{-1} \hat{e}^{n+1} , \tau_2^{n+1} \rangle ,
  \label{convergence-2}
\end{align}
in which summation by parts formulas have been repeatedly applied.
For the left hand side term associated with the artificial regularization, the following identity is valid:
\begin{equation}
\label{convergence-3}
  2A  \langle \hat{e}^{n+1} - \hat{e}^n, \hat{e}^{n+1} \rangle
  = A  ( \| \hat{e}^{n+1} \|_2^2 - \| \hat{e}^n \|_2^2
   + \| \hat{e}^{n+1} - \hat{e}^n \|_2^2 ) .
\end{equation}
The right hand side term associated with the truncation error could be bounded in a straightforward way:
\begin{equation}
  2 \langle (-\Delta_N)^{-1} \hat{e}^{n+1} , \tau_2^{n+1} \rangle
  \le 2 \| \hat{e}^{n+1} \|_{-1,N} \cdot \| \tau_2^{n+1} \|_{-1,N}
  \le  \| \hat{e}^{n+1} \|_{-1,N}^2 + \| \tau_2^{n+1} \|_{-1,N}^2  .
\end{equation}
For the second linear term on the right hand side, a direct application of Cauchy inequality gives
\begin{equation}
\label{convergence-5}
  2  \langle \hat{e}^n , \hat{e}^{n+1} \rangle
  \le \| \hat{e}^n \|_2^2 + \| \hat{e}^{n+1} \|_2^2 .
\end{equation}

For the nonlocal linear term on the right had side, we begin with a rewritten form:
\begin{align}
- 2 \varepsilon^2 \langle {\cal L}_N \hat{e}^{n+1} , \hat{e}^{n+1} \rangle
& = - 2 \varepsilon^2 \langle (J \os 1) \hat{e}^{n+1} - J \os \hat{e}^{n+1} ,
   \hat{e}^{n+1} \rangle   \nonumber \\
& =
    - 2 \varepsilon^2 (J \os 1) \| \hat{e}^{n+1} \|_2^2
   + 2 \varepsilon^2 \langle J \os \hat{e}^{n+1} ,
    \hat{e}^{n+1} \rangle  .  \label{convergence-6-1}
\end{align}
Meanwhile, for the term $2 \varepsilon^2 \langle J \os \hat{e}^{n+1} ,    \hat{e}^{n+1} \rangle$,
we apply \eqref{lem 1:0} in Lemma \ref{lem:1} and obtain
\begin{align}
2 \varepsilon^2 \langle J \os \hat{e}^{n+1} ,  \hat{e}^{n+1} \rangle
&= - 2 \varepsilon^2 \langle J \os \hat{e}^{n+1} ,
   \Delta_N ( (-\Delta_N)^{-1} \hat{e}^{n+1} ) \rangle  \nonumber\\
&\le
    \frac{\gamma_0}{2} \| \hat{e}^{n+1} \|^2_2 + \frac{C_3}{\gamma_0}
   \| \nabla_N (-\Delta_N)^{-1} \hat{e}^{n+1} \|^2_2  \nonumber\\
&\le
    \frac{\gamma_0}{2} \| \hat{e}^{n+1} \|^2_2 + \frac{C_3}{\gamma_0}
   \| \hat{e}^{n+1} \|_{-1,N}^2 ,  \label{convergence-6-2}
\end{align}
with $C_3$ only depends on $C_2$ and $\varepsilon$.
Subsequently, a combination of \eqref{convergence-6-1} and \eqref{convergence-6-2} yields
\begin{equation}
\label{convergence-6-3}
  - 2 \varepsilon^2 \langle {\cal L}_N \hat{e}^{n+1} , \hat{e}^{n+1} \rangle
  \le - 2 \varepsilon^2 (J \os 1) \| \hat{e}^{n+1} \|_2^2
  + \frac{\gamma_0}{2} \| \hat{e}^{n+1} \|^2_2 + \frac{C_3}{\gamma_0}
   \| \hat{e}^{n+1} \|_{-1,N}^2 .
\end{equation}

For the nonlinear inner product on the right hand side of \eqref{convergence-2}, we begin with a rewritten form:
\begin{equation}
\label{convergence-7-1}
- 2  \langle ( \hat{U}^n )^3 - (u^n)^3 , \hat{e}^{n+1} \rangle
= - 2  \langle ( \hat{U}^n )^3 - (u^n)^3 , \hat{e}^n \rangle
- 2  \langle ( \hat{U}^n )^3 - (u^n)^3 , \hat{e}^{n+1} - \hat{e}^n \rangle .
\end{equation}
Because of the following nonlinear expansion
\begin{equation}
( \hat{U}^n )^3 - (u^n)^3 = ( ( \hat{U}^n )^2 + \hat{U}^n u^n + (u^n)^2 ) \hat{e}^n ,
\end{equation}
we see that the first term on the right hand side of \eqref{convergence-7-1} is always non-positive:
\begin{equation}
\label{convergence-7-3}
- 2  \langle ( \hat{U}^n )^3 - (u^n)^3 , \hat{e}^n \rangle
= - 2  \langle ( \hat{U}^n )^2 + \hat{U}^n u^n  + (u^n)^2 ,   ( \hat{e}^n )^2 \rangle \le 0 .
\end{equation}
The other term on the right hand side of \eqref{convergence-7-1} could be represented as
\begin{equation}
- 2  \langle ( \hat{U}^n )^3 - (u^n)^3 , \hat{e}^{n+1} - \hat{e}^n \rangle
= - 2  \langle ( ( \hat{U}^n )^2 + \hat{U}^n u^n  + (u^n)^2 ) \hat{e}^n ,  \hat{e}^{n+1} - \hat{e}^n \rangle  .
\end{equation}
On the other hand, the $\| \cdot \|_\infty$ estimate \eqref{consistency-4-3} for $\hat{U}$
and the a-priori bound \eqref{a priori-2} have implied that
\begin{equation}
\label{convergence-7-5-0}
\| \hat{U}^n \|_\infty \le M_0 , \quad \| u^n \|_\infty \le M_0 .
\end{equation}
These facts yield the following estimate
\begin{equation}
  \| ( \hat{U}^n )^2 + \hat{U}^n u^n + (u^n)^2 \|_\infty \le 3 M_0^2 .
\end{equation}
In turn, we obtain the following inequality
\begin{align}
- 2  \langle ( \hat{U}^n )^3 - (u^n)^3 , \hat{e}^{n+1} - \hat{e}^n \rangle
& \le 2  \| ( \hat{U}^n )^2 + \hat{U}^n u^n + (u^n)^2 \|_\infty
  \cdot \| \hat{e}^n \|_2 \cdot \|  \hat{e}^{n+1} - \hat{e}^n \|_2 \nonumber\\
& \le
  6 M_0^2 \| \hat{e}^n \|_2 \cdot \|  \hat{e}^{n+1} - \hat{e}^n \|_2  \nonumber\\
& \le
   \frac{\gamma_0}{2} \| \hat{e}^n \|_2^2
   + \frac{18 M_0^4}{\gamma_0} \|  \hat{e}^{n+1} - \hat{e}^n \|_2^2 .    \label{convergence-7-6}
\end{align}
As a consequence, a substitution of \eqref{convergence-7-3} and \eqref{convergence-7-6} into \eqref{convergence-7-1} gives
\begin{equation}
\label{convergence-7-7}
- 2  \langle ( \hat{U}^n )^3 - (u^n)^3 , \hat{e}^{n+1} \rangle
\le \frac{\gamma_0}{2} \| \hat{e}^n \|_2^2 + \frac{18 M_0^4}{\gamma_0} \|  \hat{e}^{n+1} - \hat{e}^n \|_2^2  .
\end{equation}

Therefore, a substitution of \eqref{convergence-3}-\eqref{convergence-5},
\eqref{convergence-6-3} and \eqref{convergence-7-7} into \eqref{convergence-2} results in
\begin{align}
& \quad~
 \| \hat{e}^{n+1} \|_{-1,N}^2 - \| \hat{e}^n \|_{-1,N}^2
 + A  \dt ( \| \hat{e}^{n+1} \|_2^2 - \| \hat{e}^n \|_2^2 )
   + \Bigl( A - \frac{18 M_0^4}{\gamma_0} \Bigr) \dt \| \hat{e}^{n+1} - \hat{e}^n \|_2^2  \nonumber\\
& \qquad
   +  \Bigl( 2 \varepsilon^2 (J \os 1) -1 - \frac{\gamma_0}{2} \Bigr)
    \dt \| \hat{e}^{n+1} \|_2^2  \nonumber\\
& \le
       ( 1 + \frac{\gamma_0}{2} )  \dt \| \hat{e}^n \|_2^2
     + ( 1 + \frac{C_3}{\gamma_0} ) \dt \| \hat{e}^{n+1} \|_{-1,N}^2 + \dt \| \tau_2^{n+1} \|_{-1,N}^2 . 
\end{align}
Under the constraint \eqref{condition-A-1} for the regularization parameter $A$,
and making use of  the diffusivity condition \eqref{assump_gamma}, we get
\begin{align}
& \quad~
 \| \hat{e}^{n+1} \|_{-1,N}^2 - \| \hat{e}^n \|_{-1,N}^2
 + A  \dt ( \| \hat{e}^{n+1} \|_2^2 - \| \hat{e}^n \|_2^2 )
   +  \Bigl( 1 + \frac{3 \gamma_0}{2} \Bigr)
    \dt \| \hat{e}^{n+1} \|_2^2  \nonumber\\
& \le
     ( 1 + \frac{\gamma_0}{2} )  \dt \| \hat{e}^n \|_2^2
     + ( 1 + \frac{C_3}{\gamma_0} ) \dt \| \hat{e}^{n+1} \|_{-1,N}^2 + \dt \| \tau_2^{n+1} \|_{-1,N}^2 .  
\end{align}
Subsequently, an application of discrete Gronwall inequality results in the desired convergence estimate:
\begin{equation}
\label{convergence-8-3}
   \| \hat{e}^{n+1} \|_{-1,N}  + \Bigl( \gamma_0 \dt   \sum_{k=1}^{n+1} \| \hat{e}^k \|_2^2 \Bigr)^{1/2} \le C^* ( \dt^3 + h^m) ,
\end{equation}
due to the fact that $\| \tau_2^k \|_{-1,N} \le C (\dt^3 + h^m)$, for $k \le n+1$.

Moreover, we have to recover the a-priori assumption \eqref{a priori-1} at time instant $t_{n+1}$,
so that the analysis could be carried out in the induction style.
An application of an inverse inequality to the convergence estimate \eqref{convergence-8-3} implies that
\begin{align}
& \| \hat{e}^{n+1} \|_\infty \le \frac{C \| \hat{e}^{n+1} \|_{-1,N} }{h^{2}}
\le  \frac{C C^* (\dt^3 + h^m) }{h^{2}}
    \le   \frac{C' C^* (h^3 + h^m) }{h^{2}}
    \le   \frac{C_4 C^* h^3 }{h^{2}}
    =
   C_4 C^* h \le \frac12 , \nonumber\\
&
    \text{provided that $h \le  \frac{1}{2 C_4 C^*}  $} , 
\end{align}
in which we have used the linear refinement path constraint $\dt \le C h$, as well as the fact that $m \ge 3$.
This completes the error estimate for $\hat{e}$,
the numerical error between the numerical solution $\phi$ and the constructed approximation solution $\hat{U}$.

Finally, the error estimate \eqref{convergence-0} is a direct consequence of the following identity
\begin{equation}
e^k = \hat{e}^k - \dt U^1_{\dt} - \dt^2 U^2_{\dt} ,
\end{equation}
which comes from the construction \eqref{consistency-1}, as well as the fact that
\begin{equation}
 \| (U^1_{\dt} )^k \|_2 \le C ,  \quad  \| ( U^2_{\dt} )^k \|_2 \le C ,  \quad
 \text{for any $k \ge 0$} .
\end{equation}
The proof for Theorem \ref{thm:convergence-Hm1} is completed.

\subsection{Theoretical justification of the energy stability}

It has been proved in \cite{du18} that
the energy stability for the numerical scheme \eqref{scheme-1st-1} is valid under the condition \eqref{condition-A-0}.
In addition, the convergence analysis reveals that
the $\| \cdot \|_\infty$ bound \eqref{convergence-7-5-0} for the numerical solution is available
as long as another constraint~\eqref{condition-A-1} for $A$ is valid,
so that the convergence analysis could pass through.
The following corollary provides a theoretical justification of the energy stability.

\begin{corollary}
Under the assumptions of Theorem \ref{thm:convergence-Hm1},
the energy stability, namely, $E_N (u^{n+1}) \le E_N (u^n)$, is valid,
under the following constraint for the regularization parameter $A$: 	
\begin{equation}
A \ge \max \Bigl\{ \frac{18 M_0^4}{\gamma_0} , \frac32 M_0^2 - \frac12 \Bigr\} ,  \quad \mbox{with} \, \, \,
M_0 =  1 + \max_{1\le k \le N_k} \| u_N^k \|_\infty .
\end{equation}
\end{corollary}

\section{Concluding remarks}

In this work, we present detailed error estimates
for a first order stabilized semi-implicit numerical scheme for the nonlocal Cahn--Hilliard equation,
where the Fourier pseudo-spectral method is used for the spatial discretization.
We consider the discrete $H^{-1}$ norm for the error function to establish the convergence result, which avoids the complicated analysis on the nonlinear term.
In order to bound the error function in the $\ell^\infty$ norm,
we combine the standard technique for the convergence analysis with a higher order consistency estimate
to ensure the convergence order high enough to use the inverse inequality.
As a result of the $\ell^\infty$ boundness of the error function,
we derive the uniform $\ell^\infty$ bound of the numerical solution,
and then, the energy stability of the numerical scheme, obtained in \cite{du18},
is improved by requiring a new assumption on the stabilizer.

It is worth mentioning that
we use the higher order consistency analysis to pick up only the temporal truncated error
since the spatial spectral accuracy ${O}(h^m)$ is sufficient as long as $m$ is large enough.
However, if one considers the lower order spatial approximations,
for instance, the finite difference and finite element methods,
the truncated error is usually of the order two
and the higher order consistency estimate is also necessary to pick up the spatial truncated error,
see, e.g., \cite{guan17a,guan14a} and references therein.

\section*{Acknowledgments}
X. Li's work is partially supported by NSFC grant 11801024.
Z. Qiao's work is partially supported by the Hong Kong Research Council GRF grants 15325816 and 15300417.
C. Wang's work is partially supported by NSF grant NSF DMS-1418689.

\bibliographystyle{plain}
\bibliography{draft0}

\end{document}